\documentclass{article}
\usepackage[utf8]{inputenc}
\usepackage{amsmath,amssymb,dsfont,amsfonts}
\usepackage{amsthm}
\usepackage{mathrsfs}
\usepackage{graphics}
\usepackage{graphicx}
\usepackage{color}

\newtheorem{example}{Example}

\newtheorem{theorem}{Theorem}[section]
\newtheorem{proposition}[theorem]{Proposition}
\newtheorem{corollary}[theorem]{Corollary}
\newtheorem{remark}[theorem]{Remark}
\newtheorem{definition}[theorem]{Definition}

\usepackage[bottom=2cm,top=2.5cm,left=2.5cm,right=2.5cm]{geometry}

\title{The Markov Accessibility Topology and its Probabilistic Structure}
\author{Lourival Lima\footnote{Mathematics Department, Escuela Superior Politécnica del Litoral, ESPOL. Vía Perimetral 5, Guayaquil, Ecuador. E-mail: lourodri@espol.edu.ec. 
Research supported by ESPOL.} \ \ and\ \ Jorge Vielma\footnote{Mathematics Department, Escuela Superior Politécnica del Litoral, ESPOL. Vía Perimetral 5, Guayaquil, Ecuador. E-mail: jevielma@espol.edu.ec 
Research supported by ESPOL.}}

\date{}

\begin{document}

\maketitle

\begin{abstract}

 In this paper, we associate a canonical topology with the transition operator of a discrete-time Markov chain. We prove that this construction defines an Alexandroff topology, called the Markov Accessibility Topology, and subsequently establish that it coincides with the Alexandroff topology induced by the accessibility relation. This establishes a direct connection between probabilistic accessibility and order-theoretic topology, allowing several classical notions from Markov chain theory to be interpreted in topological terms. Among other results, we characterize absorbing subsets as closed sets, relate irreducibility to topological connectedness and the triviality of the topology, describe minimal open neighborhoods through accessibility and hitting probabilities, and investigate the communication quotient together with its relationship to lumpability. 
\end{abstract}


\medskip


\vspace{0.5cm}

\section{Introduction}

Discrete-time Markov chains provide a natural mathematical framework for describing systems that evolve randomly over time. Besides their numerous applications, they have long been one of the central objects of probability theory, giving rise to a rich mathematical structure and a vast literature \cite{K.L.Chung,KemenySnell,J.R.Norris,LevinPeresWilmer,E.Seneta}. Many of the qualitative properties of a Markov chain are determined not by the actual transition probabilities, but rather by the way its states are connected through the accessibility relation. This relation identifies which states can be reached from one another after finitely many transitions and forms the basis of fundamental concepts such as communication classes, irreducibility, recurrence, and absorbing subsets \cite{J.R.Norris,LevinPeresWilmer,E.Seneta}.

Although accessibility has been extensively studied from a probabilistic perspective, it also suggests the existence of an underlying order-theoretic structure on the state space. Since Alexandroff spaces are in one-to-one correspondence with preorders \cite{ShiraziGolestani,Echi,Pajoohesh}, it is natural to ask whether the transition mechanism of a Markov chain determines an intrinsic topology reflecting its accessibility structure.

Despite this natural connection, the accessibility structure of Markov chains has rarely been investigated from a genuinely topological viewpoint. While topological methods have proved successful in several areas including dynamical systems, stochastic analysis, stochastic control, and topological structures induced by algebraic actions \cite{Echi,LimaRuffino,LimaRuffinoSouza,Mejias-Vielma-Aponte-Lima}, the topology intrinsically encoded by the transition mechanism of a Markov chain has not, to the best of our knowledge, been systematically studied. Consequently, many classical notions from Markov chain theory have not yet been interpreted through the language of topology.

The main purpose of this paper is to show that the transition mechanism of a Markov chain intrinsically determines a canonical topological structure on the state space. Rather than beginning with an order-theoretic object and invoking the classical correspondence between preorders and Alexandroff spaces, we start directly from the probabilistic dynamics of the chain. More precisely, using the one-step transition operator, we construct a family of subsets of the state space and prove that it defines an Alexandroff topology, which we call the Markov Accessibility Topology. This construction is canonical, as it is determined solely by the transition operator of the chain. Only afterwards do we prove that the accessibility relation is in fact a preorder and that the resulting topology coincides with the Alexandroff topology associated with this preorder. Thus, the connection with order-theoretic topology appears as a structural consequence of the construction rather than as its starting point.

The significance of this construction lies in the fact that the topology is intrinsic to the transition dynamics of the chain. It is not imposed externally through an \emph{a priori} order or topological structure, but instead emerges directly from the probabilistic behavior encoded by the transition operator. This natural construction establishes a bridge between probabilistic accessibility and order-theoretic topology, allowing several classical concepts from Markov chain theory to be interpreted in purely topological terms.

We prove that absorbing subsets correspond exactly to closed sets, relate irreducibility to connectedness and to the triviality of the topology, characterize minimal open neighborhoods in terms of accessibility and hitting probabilities, establish a pathwise characterization of open sets, and investigate the quotient topology induced by communication classes together with its relationship to lumpability and quotient Markov chains. These results reveal that accessibility is not merely a relation describing possible transitions between states. Rather, it determines an intrinsic topological structure governing many qualitative features of discrete-time Markov chains.

The remainder of this paper is organized as follows. Section~2 reviews the necessary background on Markov chains, preorders, and Alexandroff spaces. Section~3 introduces the Markov Accessibility Topology through a probabilistic construction, proves that it is an Alexandroff topology, and subsequently establishes that it coincides with the Alexandroff topology associated with the accessibility relation. Section~4 develops the local topological theory associated with this construction, providing topological characterizations of several classical notions from Markov chain theory, including irreducibility, absorbing subsets, hitting times, and path accessibility. Finally, Section~5 studies the quotient topology induced by communication classes and its relationship with lumpability and quotient Markov chains.

\section{Preliminaries}

This section collects the basic concepts and notation that will be used throughout the paper. For background material on Markov chains, we refer the reader to Norris \cite{J.R.Norris}, Levin, Peres and Wilmer \cite{LevinPeresWilmer}, and Kemeny and Snell \cite{KemenySnell}. For the theory of Alexandroff spaces and their connections with order-theoretic structures, we refer the reader to Shirazi and Golestani \cite{ShiraziGolestani}, Echi \cite{Echi}, Pajoohesh \cite{Pajoohesh}, and Echi and Turki \cite{EchiTurki}. These concepts provide the probabilistic and topological foundations for the results developed in the subsequent sections.

\subsection{Markov chains}

Let $(\Omega,\mathcal F,\mathbb P)$ be a probability space and let $(X_n)_{n\geq 0}$ be a stochastic process taking values in a finite or countable state space \(X\). The process \((X_n)_{n\geq0}\) is called a \emph{discrete-time Markov chain} if, for every \(n\geq0\) and every choice of states
\[
x_0,x_1,\ldots,x_n,x,y\in X,
\]
the following Markov property holds:
\[
\mathbb P(X_{n+1}=y\,|\,X_n=x,X_{n-1}=x_{n-1},\ldots,X_0=x_0)
=
\mathbb P(X_{n+1}=y\,|\,X_n=x).
\]
In other words, conditioned on the present state, the future evolution of the process is independent of its past history. The transition probabilities are defined by
\[
P(x,y)
=
\mathbb P(X_{n+1}=y\,|\,X_n=x),
\qquad x,y\in X.
\]
The matrix
\[
P=(P(x,y))_{x,y\in X}
\]
is called the \emph{transition matrix} of the chain. Since \(P\) is stochastic, we have
\[
P(x,y)\geq0,
\qquad
\sum_{y\in X}P(x,y)=1,
\]
for every \(x\in X\). For \(n\geq1\), the \(n\)-step transition probabilities are defined by
\[
P^n(x,y)
=
\mathbb P(X_n=y\,|\,X_0=x).
\]
These probabilities satisfy the Chapman--Kolmogorov equations:
\[
P^{n+m}(x,y)
=
\sum_{z\in X}
P^n(x,z)P^m(z,y),
\]
for all \(m,n\in\mathbb N_0\) and all \(x,y\in X\). A probability measure \(\mu\) on the state space \(X\) is called an
\emph{initial distribution} of the Markov chain if

\[
\mathbb P(X_0=x)=\mu(x),
\qquad x\in X.
\]
The probability law of the chain corresponding to the initial
distribution \(\mu\) will be denoted by \(\mathbb P_\mu\), and the
associated expectation by \(\mathbb E_\mu\). For \(x\in X\), the Dirac measure at \(x\) is denoted by \(\delta_x\) and is defined by
\[
\delta_x(y)=
\begin{cases}
1,& y=x,\\
0,& y\neq x.
\end{cases}
\]
Thus, under \(\mathbb P_{\delta_x}\), the chain starts from the state \(x\) with probability one. A property is said to hold \emph{\(\mathbb P_\mu\)-almost surely} if it holds with \(\mathbb P_\mu\)-probability one. Throughout the paper, the abbreviation \emph{a.s.} stands for \emph{almost surely}.

\bigskip

Let $\displaystyle{\Omega=X^{\mathbb N_0}}$ be the path space of the chain. The shift operator
\[
\theta:\Omega\to\Omega
\]
is defined by
\[
(\theta\omega)_n=\omega_{n+1},
\qquad n\ge0.
\]
More generally,
\[
(\theta_k\omega)_n=\omega_{n+k},
\qquad k,n\ge0.
\]
Thus, \(\theta_k\) removes the first \(k\) states of a trajectory and views the process from time \(k\) onward. In this context,  $(X,P)$ represents a discrete-time Markov chain with state space $X$ and transition matrix $\displaystyle{P=(P(x,y))_{x,y\in X}}$.

\bigskip

\begin{definition}
The accessibility relation on $X$ is the binary relation $\preceq$ defined by
\[
x\preceq y
\quad \text{if and only if} \quad
P^n(x,y)>0
\]
for some \(n\in\mathbb N_0\).
\end{definition}
In other words, \(x\preceq y\) whenever there exists a path from \(x\) to \(y\) that can be followed by the Markov chain with positive probability in finitely many transitions. As a consequence of the Chapman--Kolmogorov equations, the accessibility relation possesses a natural order structure. This fact will be established in Theorem~\ref{T1}.

\begin{definition}
Two states \(x,y\in X\) are said to \emph{communicate} if
\[
x\sim y
\quad \text{if and only if} \quad
x\preceq y
\ \text{and}\
y\preceq x.
\]
The equivalence classes induced by the relation \(\sim\) are called the \emph{communication classes} of the Markov chain.
\end{definition}

Communication classes capture the global accessibility structure of a Markov chain. In particular, they allow us to distinguish between chains whose states are all mutually accessible and those that decompose into several disconnected communication classes. This distinction is formalized through the notions of irreducibility and reducibility.

\begin{definition}
A discrete-time Markov chain \((X,P)\) is said to be \emph{irreducible} if all states belong to the same communication class. Equivalently,
\[
x\sim y,
\qquad
\forall\, x,y\in X.
\]
\end{definition}

While accessibility and communication describe the possibility of moving through the state space, it is equally important to identify regions from which the chain cannot escape. Such structures are captured by the notions of absorbing states and absorbing subsets. As we shall see later, these probabilistic concepts admit a natural topological interpretation within the framework of the Markov accessibility topology.

\bigskip

\begin{definition}
A state \(x\in X\) is said to be \emph{absorbing} if
\[
P(x,x)=1.
\]
Equivalently, once the chain enters the state \(x\), it remains there forever.
\end{definition}

A natural generalization of the notion of an absorbing state is that of an absorbing subset of the state space.

\begin{definition}
A subset \(C\subseteq X\) is said to be \emph{absorbing} if
\[
P(x,C)=1,
\qquad
\forall\,x\in C.
\]
Equivalently, once the chain enters \(C\), it cannot leave \(C\).
\end{definition}

\bigskip

The notion of an absorbing subset provides a natural link between the probabilistic dynamics of a Markov chain and the topology introduced in this work. Indeed, one of our main results establishes that the absorbing subsets are precisely the closed sets of the Markov accessibility topology.

\subsection{Alexandroff spaces}
The accessibility relation introduced in the previous subsection naturally leads to an order-theoretic structure on the state space. Since preorders and Alexandroff topologies are closely related, we briefly recall some basic notions concerning Alexandroff spaces that will be needed in the sequel.

\begin{definition}
A topological space \((X,\tau)\) is called an \emph{Alexandroff space} if arbitrary intersections of open sets are open.
\end{definition}

Alexandroff spaces occupy a distinguished position in topology due to their close relationship with order-theoretic structures. In particular, every preorder induces an Alexandroff topology, and conversely every Alexandroff topology determines a canonical preorder.

\bigskip

\begin{proposition}
Let $(X,\tau)$ be an Alexandroff space. Then, for every
\(x\in X\), there exists a smallest open set containing \(x\),
namely
\[
U_x
=
\bigcap\{U\in\tau:x\in U\}.
\]
\end{proposition}

The collection
\[
\mathcal B=\{U_x:x\in X\}
\]
forms a basis for the topology \(\tau\).

\begin{definition}
Let \((X,\tau)\) be a topological space. The \emph{specialization preorder} on \(X\) is defined by
\[
x\leq_\tau y
\]
if and only if every open set containing \(x\) also contains \(y\).
\end{definition}

\bigskip

\begin{proposition}
Let $(X,\tau)$ be an Alexandroff space and let
\(x\in X\). Then
\[
\overline{\{x\}}
=
\{y\in X:x\leq_\tau y\}.
\]
\end{proposition}

The following result is classical (see, for example, \cite{Echi,ShiraziGolestani}).

\begin{theorem}
There is a one-to-one correspondence between Alexandroff topologies on a set \(X\) and preorders on \(X\). More precisely, given a preorder \(\leq\) on \(X\), the associated
Alexandroff topology consists of all downward closed subsets of \(X\).
\end{theorem}

The preceding correspondence suggests that any preorder arising naturally from the dynamics of a Markov chain should induce an associated Alexandroff topology. In the next section, we show that the accessibility relation gives rise to a canonical topology on the state space, which we call the Markov accessibility topology.

\section{The Markov Accessibility Topology}

The transition probabilities of a Markov chain naturally determine a family of interesting subsets of the state space. For a subset $U\subseteq X$, define
\[
P^{*}(U):=\{x\in X : P(x,U)>0\},
\]
where
\[
P(x,U):=\sum_{y\in U} P(x,y),
\]
and consider the following family of subsets defined by
\[
\tau_P:=\{U\subseteq X : P^{*}(U)\subseteq U\}.
\]
Equivalently, a subset $U\subseteq X$ belongs to $\tau_P$ if and only if no state outside $U$ can enter $U$ with positive probability in one step.

\bigskip

\begin{theorem}
The family
\[
\tau_P=\{U\subseteq X:P^*(U)\subseteq U\}
\]
is an Alexandroff topology on \(X\).
\end{theorem}

\begin{proof}
We first verify that $\tau_P$ is a topology on $X$. Since $P^*(\varnothing)=\varnothing$, it follows that $\varnothing\in\tau_P$. Moreover, since $P^*(X)\subseteq X$, we also have $X\in\tau_P$. Let $\{U_i\}_{i\in I}$ be an arbitrary family of elements of $\tau_P$.
Using the monotonicity of $P^*$, we obtain
\[
P^*\left(\bigcup_{i\in I}U_i\right)
=
\bigcup_{i\in I}P^*(U_i).
\]
Since $P^*(U_i)\subseteq U_i$ for every $i\in I$, it follows that
\[
P^*\left(\bigcup_{i\in I}U_i\right)
\subseteq
\bigcup_{i\in I}U_i.
\]
Therefore, $\displaystyle{\bigcup_{i\in I}U_i\in\tau_P}$. Now let $\{U_i\}_{i\in I}$ be an arbitrary family of elements of
$\tau_P$. Since $\displaystyle{\bigcap_{i\in I}U_i\subseteq U_j}$ for every $j\in I$, the monotonicity of $P^*$ yields
\[
P^*\left(\bigcap_{i\in I}U_i\right)
\subseteq P^*(U_j)
\]
for every $j\in I$. Hence,
\[
P^*\left(\bigcap_{i\in I}U_i\right)
\subseteq
\bigcap_{i\in I}P^*(U_i).
\]
Since $P^*(U_i)\subseteq U_i$ for every $i\in I$, we conclude that
\[
P^*\left(\bigcap_{i\in I}U_i\right)
\subseteq
\bigcap_{i\in I}U_i.
\]
Thus, $\displaystyle{\bigcap_{i\in I}U_i\in\tau_P}$. Consequently, $\tau_P$ is a topology closed under arbitrary
intersections, and therefore it is an Alexandroff topology on $X$.
\end{proof}

\bigskip

From now on, we refer to $\tau_P$ as the Markov accessibility topology associated with the transition matrix $P$. The next theorem establishes the precise relationship between the probabilistically constructed topology and the accessibility relation of the Markov chain. In particular, it shows that the accessibility relation is a preorder and that the Markov accessibility topology coincides with the Alexandroff topology associated with this preorder. This characterization provides the foundation for all subsequent developments.

\bigskip

\begin{theorem}
\label{T1}
Let $(X,P)$ be a discrete-time Markov chain with state space $X$. Then $\preceq$ is a preorder on $X$, and $\tau_P$ coincides with the Alexandroff topology associated with $\preceq$. 
\end{theorem}

\begin{proof}
We first show that the relation $\preceq$ is a preorder on $X$. It is Reflexive, in fact,
for every $x\in X$, $P^0(x,x)=1>0$, where $P^0$ denotes the identity transition matrix. Hence, $x\preceq x$.
It is Transitive.
Suppose that $x\preceq y
\quad \text{and} \quad
y\preceq z$.
Then there exist integers $m,n\in\mathbb N_0$ such that
\[
P^m(x,y)>0
\quad \text{and} \quad
P^n(y,z)>0.
\]
By the Chapman--Kolmogorov equations (cf. Norris~\cite[Section~1.1]{J.R.Norris}),
\[
P^{m+n}(x,z)
=
\sum_{w\in X} P^m(x,w)P^n(w,z).
\]
Since the term corresponding to $w=y$ is strictly positive, we obtain
\[
P^{m+n}(x,z)\ge P^m(x,y)P^n(y,z)>0.
\]
Hence, $x\preceq z$. Therefore, $\preceq$ is transitive, and consequently it is a preorder on $X$. 

\bigskip

Now we prove that $\tau_P$ coincides with the Alexandroff topology associated with $\preceq$. Recall that $U\in\tau_P$ implies
$P^{*}(U)\subseteq U,
$
that is, $P(x,U)>0, \ \ \text{then} \   x\in U$.
We claim that this is equivalent to $U$ being downward closed with respect to $\preceq$. In fact, assume that $U\in\tau_P$, and let $x,y\in X$ such that $x\preceq y
\quad \text{and} \quad
y\in U$. By definition of $\preceq$, there exists $n\in\mathbb N_0$ such that $\displaystyle{P^n(x,y)>0}$. Hence,
\[
P^n(x,U)\ge P^n(x,y)>0.
\]
We prove by induction that
\[
(P^n)^{*}(U)\subseteq U
\qquad \forall n\in\mathbb N_0.
\]
For $n=1$, this is precisely the definition of $\tau_P$. Assume the result holds for some $n\ge1$, and let $\displaystyle{x\in (P^{n+1})^{*}(U)}$, then $\displaystyle{P^{n+1}(x,U)>0}$. Applying the Chapman--Kolmogorov equations,
\[
P^{n+1}(x,U)
=
\sum_{w\in X} P(x,w)P^n(w,U).
\]
Therefore, there exists $w\in X$ such that
\[
P(x,w)>0
\quad \text{and} \quad
P^n(w,U)>0.
\]
Thus,
\[
w\in (P^n)^{*}(U)\subseteq U.
\]
Since $P(x,w)>0$ and $w\in U$, we obtain $\displaystyle{P(x,U)>0}$, which implies
\[
x\in P^{*}(U)\subseteq U.
\]
Hence,
\[
(P^{n+1})^{*}(U)\subseteq U.
\]
By induction,
\[
(P^n)^{*}(U)\subseteq U
\qquad \forall n\in\mathbb N_0.
\]
Since $P^n(x,U)>0$, it follows that
\[
x\in (P^n)^{*}(U)\subseteq U.
\]
Therefore, $U$ is downward closed with respect to $\preceq$. 

\bigskip

Conversely, suppose that $U\subseteq X$ is downward closed with respect to $\preceq$. Let $x\in X$ such that $\displaystyle{P(x,U)>0}$. Then there exists $y\in U$ with $\displaystyle{P(x,y)>0}$. Hence, $x\preceq y$. Since $U$ is downward closed, it follows that $x\in U$. Therefore,
\[
P^{*}(U)\subseteq U,
\]
which proves that $U\in\tau_P$. Consequently, $\tau_P$ coincides with the Alexandroff topology associated with the preorder $\preceq$.
\end{proof}

\bigskip

\begin{theorem}
Let $(X,P)$ be a discrete-time Markov chain, and let $\tau_P$ be the Markov accessibility topology on $X$. If the chain is irreducible, then the topological space $(X,\tau_P)$ is connected.
\end{theorem}

\begin{proof}
We prove the contrapositive. Suppose that $(X,\tau_P)$ is not connected. Then there exist nonempty disjoint open sets $U,V\in\tau_P$ such that $X=U\cup V$ and
$U\cap V=\varnothing$.
 Let $u\in U$ and $v\in V$. We claim that $u\not\preceq v$. Indeed, suppose by contradiction that $u\preceq v$. Since $v\in V$ and $V$ is open in $\tau_P$, by the Alexandroff characterization of $\tau_P$, the set $V$ is downward closed with respect to $\preceq$. Hence, from $u\preceq v$ and $v\in V$, we obtain $u\in V$. This contradicts $u\in U$ and $U\cap V=\varnothing$. Therefore, $u\not\preceq v$.
Equivalently,
\[
P^n(u,v)=0
\qquad\text{for all } n\in\mathbb N_0.
\]
Thus, $v$ is not accessible from $u$, and so the chain is not irreducible.
\end{proof}

\bigskip

\begin{example}
Observe that connectedness does not imply irreducibility. Indeed, consider the two-state Markov chain
\[
1\to 2,
\]
whose Markov primal topology is
\[
\tau_P=\{\varnothing,\{1\},\{1,2\}\}.
\]
This topology is connected, although the chain is not irreducible, since state $1$ is not accessible from state $2$.
\end{example}

Observe that when every state communicates with every other state, the accessibility preorder collapses to the universal relation, and the associated topology acquires a remarkably simple form. The following theorem makes this correspondence precise.

\begin{theorem}
The following statements are equivalent:

\begin{enumerate}
    \item $(X,P)$ is irreducible;
    
    \item the accessibility preorder $\preceq$ satisfies
    \[
    x\preceq y
    \qquad \text{for all } x,y\in X;
    \]
    
    \item $\tau_P$ is the trivial topology on $X$, that is,
    \[
    \tau_P=\{\varnothing,X\}.
    \]
\end{enumerate}
\end{theorem}
\begin{proof}
The equivalence between $(1)$ and $(2)$ follows directly from the definition of irreducibility. We now prove that $(2)$ implies $(3)$. Assume that
\[
x\preceq y
\qquad \text{for all } x,y\in X.
\]
Let \(U\in\tau_P\) be nonempty. Choose \(y\in U\). Since \(x\preceq y\) for every \(x\in X\), and since \(U\) is open, it follows from the Alexandroff characterization of \(\tau_P\) that \(U\) is downward closed with respect to \(\preceq\). Hence,
\[
x\in U
\qquad \text{for every } x\in X.
\]
Therefore, $U=X$. Thus the only open subsets of \(X\) are \(\varnothing\) and \(X\), and so $\tau_P=\{\varnothing,X\}$. It remains to prove that $(3)$ implies $(2)$. Assume that
\[
\tau_P=\{\varnothing,X\}.
\]
Suppose, by contradiction, that there exist \(x,y\in X\) such that
\[
x\not\preceq y.
\]
Consider the set
\[
U_y:=\{z\in X:z\preceq y\}.
\]
By reflexivity of \(\preceq\), we have \(y\in U_y\), so \(U_y\neq\varnothing\). Moreover, since \(x\not\preceq y\), we have \(x\notin U_y\), and hence \(U_y\neq X\). We claim that \(U_y\in\tau_P\). Indeed, let \(a,b\in X\) with
\[
a\preceq b
\qquad \text{and} \qquad
b\in U_y.
\]
Since \(b\in U_y\), we have \(b\preceq y\). By transitivity of \(\preceq\),
\[
a\preceq y.
\]
Thus \(a\in U_y\). Therefore \(U_y\) is downward closed with respect to \(\preceq\), and by the Alexandroff characterization of \(\tau_P\), Theorem \ref{T1}, it follows that $U_y\in\tau_P$. Hence \(U_y\) is a nonempty proper open subset of \(X\) which is a contradiction.
Therefore,
\[
x\preceq y
\qquad \text{for all } x,y\in X.
\]
\end{proof}

\bigskip

A natural question is how closed sets in the Markov primal topology reflect the dynamics of the underlying Markov chain. As it turns out, closed subsets are exactly those from which the chain cannot escape.

\begin{theorem}
A subset $C\subseteq X$ is closed in $(X,\tau_P)$ if and only if
\[
P(x,C)=1
\qquad \text{for all } x\in C.
\]
\end{theorem}
\begin{proof}
Suppose first that $C$ is closed in $(X,\tau_P)$. Then its complement $U:=X\setminus C$ is open in $\tau_P$. Hence, by the definition of the Markov primal topology, $\displaystyle{P^{*}(U)\subseteq U}$. Let $x\in C$. Since $x\notin U$, the inclusion above implies that $x\notin P^{*}(U)$. Therefore, $\displaystyle{P(x,U)=0}$. Since $U=X\setminus C$, we have $\displaystyle{P(x,X\setminus C)=0}$. Consequently, $\displaystyle{P(x,C)=1}$. Thus,
\[
P(x,C)=1
\qquad \text{for all } x\in C.
\]
Conversely, assume that
\[
P(x,C)=1
\qquad \text{for all } x\in C.
\]
We prove that $C$ is closed by showing that its complement $\displaystyle{U:=X\setminus C}$ is open. Let $x\in P^{*}(U)$. Then $\displaystyle{P(x,U)>0}$. We claim that $x\in U$. Indeed, if $x\in C$, then by hypothesis
\[
P(x,C)=1,
\]
and therefore
\[
P(x,U)=P(x,X\setminus C)=0,
\]
which contradicts $P(x,U)>0$. Hence $x\notin C$, and so $x\in U$. Therefore,
\[
P^{*}(U)\subseteq U,
\]
which means that $U\in\tau_P$. Hence $X\setminus C$ is open, and consequently $C$ is closed in $(X,\tau_P)$.
\end{proof}

\bigskip

\begin{remark}
The previous theorem shows that the closed sets of the Markov accessibility topology are precisely the absorbing subsets of the chain. In particular, singleton closed sets correspond to absorbing states, while more general closed sets correspond to closed communicating structures in the classical theory of Markov chains.
\end{remark}

\bigskip

\begin{corollary}
For $x\in X$, the singleton set $\{x\}$ is closed in $(X,\tau_P)$ if and only if $x$ is an absorbing state.
\end{corollary}

\begin{proof}
By the previous theorem, a subset $C\subseteq X$ is closed in $(X,\tau_P)$ if and only if
\[
P(y,C)=1
\qquad \text{for all } y\in C.
\]
Taking $C=\{x\}$, we obtain that $\{x\}$ is closed if and only if
\[
P(x,\{x\})=1.
\]
Since
\[
P(x,\{x\})=P(x,x),
\]
it follows that
\[
P(x,x)=1,
\]
which is precisely the definition of an absorbing state.
\end{proof}

\bigskip

\begin{corollary}
If $(X,P)$ is an irreducible discrete-time Markov chain, then the only closed subsets of $(X,\tau_P)$ are $\varnothing$ and $X$.
\end{corollary}

\begin{proof}
By Theorem 3, if $(X,P)$ is irreducible, then
\[
\tau_P=\{\varnothing,X\},
\]
that is, $\tau_P$ is the trivial topology on $X$.
\end{proof}

\bigskip

The family of closed sets is closed under arbitrary intersections and finite unions, as expected from general topological principles. In the Markov setting, this corresponds to the stability of absorbing subsets under these operations.

\bigskip

\begin{example}
Let $X=\{1,2,3\}$ and consider the discrete-time Markov chain with transition matrix
\[
P=
\begin{pmatrix}
0 & 1 & 0\\
0 & 0 & 1\\
1 & 0 & 0
\end{pmatrix}.
\]
This chain corresponds to the deterministic cycle
\[
1\to2\to3\to1.
\]
For any $x,y\in X$, there exists $n\ge0$ such that
\[
P^n(x,y)>0.
\]
Hence, the chain is irreducible. By Theorem 3, the Markov accessibility topology $\tau_P$ consists only of the trivial subsets. Since every state is accessible from every other state, it follows that any nonempty downward closed subset must contain all elements of $X$. Therefore,
\[
\tau_P=\{\varnothing,X\}.
\]
\end{example}

\bigskip

\begin{example}
Let $X=\{1,2\}$ and consider the Markov chain with transition matrix
\[
P=
\begin{pmatrix}
0 & 1\\
0 & 1
\end{pmatrix}.
\]
In this case,
\[
1\to2,
\qquad
2\to2.
\]
The accessibility relation satisfies
\[
1\preceq1,
\qquad
1\preceq2,
\qquad
2\preceq2,
\]
but
\[
2\not\preceq1.
\]
Hence, the chain is not irreducible. The downward closed subsets of $X$ are
\[
\varnothing,
\qquad
\{1\},
\qquad
\{1,2\}.
\]
Therefore,
\[
\tau_P=\{\varnothing,\{1\},X\}.
\]
This topology is connected, since there do not exist two nonempty disjoint open subsets whose union is $X$. However, the chain is reducible. The closed subsets are
\[
\varnothing,
\qquad
\{2\},
\qquad
X.
\]
Indeed,
\[
P(2,\{2\})=1,
\]
so the state $2$ is absorbing.
\end{example}

\section{Minimal Open Sets and Local Probabilistic Structure}

In this section, we investigate the local structure of the Markov accessibility topology through its minimal open sets. These sets provide a topological description of probabilistic accessibility neighborhoods and play a role analogous to principal neighborhoods in Alexandroff spaces. We show that every open subset can be described in terms of minimal open neighborhoods and establish connections between local topological structure, accessibility relations, communication classes, and probabilistic dynamics of the Markov chain.

\begin{definition}
Let $(X,P)$ be a discrete-time Markov chain, and let $\preceq$ be the accessibility preorder associated with the chain. For each $x\in X$, define
\[
P_P(x):=\{y\in X:y\preceq x\}.
\]
We refer to $P_P(x)$ as the \emph{minimal open set associated with} $x$.
\end{definition}

\bigskip

\begin{theorem}
For each $x\in X$, the set $P_P(x)$ is the smallest open set in $\tau_P$ containing $x$. That is:

\begin{enumerate}
    \item $P_P(x)\in\tau_P$;
    
    \item $x\in P_P(x)$;
    
    \item if $U\in\tau_P$ and $x\in U$, then
    \[
    P_P(x)\subseteq U.
    \]
\end{enumerate}
\end{theorem}

\begin{proof}
Fix $x\in X$. We prove the three assertions separately. First, we show that $P_P(x)$ is open. By definition,
\[
P_P(x)=\{y\in X:y\preceq x\}.
\]
Let $z\in P_P(x)$ and suppose that $y\preceq z$. Since $z\in P_P(x)$, we have $z\preceq x $.
By transitivity of the accessibility preorder, $y\preceq x$. Hence, $y\in P_P(x)$. Therefore, $P_P(x)$ is downward closed with respect to $\preceq$. By the Alexandroff characterization of $\tau_P$, it follows that
\[
P_P(x)\in\tau_P.
\]
Second, since $\preceq$ is reflexive, we have $x\preceq x$. Thus,
\[
x\in P_P(x).
\]
Finally, let $U\in\tau_P$ be such that $x\in U$. Since $U$ is open, it is downward closed with respect to $\preceq$. If $y\in P_P(x)$, then
\[
y\preceq x.
\]
By the facts that $x\in U$ and $U$ is downward closed, we obtain $y\in U$. Hence,
\[
P_P(x)\subseteq U.
\]
\end{proof}

\bigskip

\begin{corollary}
The family $\mathcal{B}:=\{P_P(x):x\in X\}$ forms a basis for Markov accessibility topology $\tau_P$. In particular, for every $U\in\tau_P$,
\[
U=\bigcup_{x\in U} P_P(x).
\]
\end{corollary}

\bigskip

\begin{proposition}
For $x,y\in X$, $\displaystyle{x\preceq y
 \ \text{if and only if} \ 
P_P(x)\subseteq P_P(y)}$.

\end{proposition}

\begin{proof}
Assume that $x\preceq y$. Let $z\in P_P(x)$. Then $z\preceq x$. By transitivity of the accessibility preorder, $z\preceq y$. Hence, $z\in P_P(y)$. Therefore, $P_P(x)\subseteq P_P(y)$. Conversely, suppose that $P_P(x)\subseteq P_P(y)$. Since $\preceq$ is reflexive, $x\preceq x$, and therefore $x\in P_P(x)$. By the inclusion hypothesis, $x\in P_P(y)$. Hence, $x\preceq y$. This completes the proof.
\end{proof}

\bigskip

\begin{proposition}
For $x,y\in X$, $P_P(x)=P_P(y)
\ \text{if and only if} \ \ 
x\sim y$. Hence, the minimal open sets are constant on communication classes.
\end{proposition}

\begin{proof}
Assume first that $P_P(x)=P_P(y)$. By the previous proposition, $P_P(x)\subseteq P_P(y) \ \  \text{if and only if}\  \  x\preceq y$, and similarly, $P_P(y)\subseteq P_P(x)
 \ \ \text{if and only if} \ \  y\preceq x$. Since the two sets are equal, both inclusions hold. Therefore, $x\preceq y
 \ \text{and} \ 
y\preceq x$, which means that $x\sim y$. Conversely, suppose that $x\sim y$. Then $x\preceq y
 \ \  \text{and} \  \ 
y\preceq x$. Again by the previous proposition, $P_P(x)\subseteq P_P(y) \ \ \text{and} \ \ 
P_P(y)\subseteq P_P(x)$. Hence, $P_P(x)=P_P(y)$.
\end{proof}

\bigskip

The previous results demonstrate that the Markov accessibility topology is entirely governed by the accessibility structure of the underlying Markov chain. In particular, minimal open sets encode both the accessibility preorder and the communication classes, thereby capturing the local geometry of the topology. We now show that these sets admit a natural probabilistic characterization through hitting probabilities.

\bigskip

\begin{theorem}
\label{T:hitting}
For a subset $U\subseteq X$, define the first hitting time of $U$ by
\[
\tau_U:=\inf\{n\geq 0 : X_n\in U\}.
\]
Then the following are equivalent:

\begin{enumerate}
    \item $U\in\tau_P$;
    
    \item
    $ \displaystyle{
    \mathbb P_y(\tau_U<\infty)=0,
    \qquad \text{for every } y\in X\setminus U.
    }$
\end{enumerate}
\end{theorem}

\begin{proof}
Suppose first that \(U\in\tau_P\). We prove that $\displaystyle{ \mathbb P_y(\tau_U<\infty)=0
 \ \text{for every } y\in X\setminus U}$. Let \(y\in X\setminus U\). If, by contradiction,
$ \displaystyle{
\mathbb P_y(\tau_U<\infty)>0
}$,
then there exists \(n\in\mathbb N_0\) such that
\[
\mathbb P_y(X_n\in U)>0.
\]
Equivalently,
\[
P^n(y,U)>0.
\]
Hence, there exists \(x\in U\) such that
\[
P^n(y,x)>0.
\]
Therefore,
\[
y\preceq x.
\]
Since \(U\in\tau_P\), by the Alexandroff characterization of the Markov accessibility topology, Theorem \ref{T1},  \(U\) is downward closed with respect to \(\preceq\). Thus, from \(y\preceq x\) and \(x\in U\), we obtain
\[
y\in U,
\]
which contradicts \(y\in X\setminus U\). Hence,
\[
\mathbb P_y(\tau_U<\infty)=0.
\]
Conversely, assume that
\[
\mathbb P_y(\tau_U<\infty)=0
 \  \text{for every } y\in X\setminus U.
\]
Let \(y\in X\) and \(x\in U\) be such that
\[
y\preceq x.
\]
Then there exists \(n\in\mathbb N_0\) such that
\[
P^n(y,x)>0.
\]
Since \(x\in U\), it follows that
\[
P^n(y,U)>0.
\]
Thus,
\[
\mathbb P_y(\tau_U<\infty)>0.
\]
By the hypothesis, this is impossible if \(y\in X\setminus U\). Therefore, $y\in U$. Hence, \(U\) is downward closed with respect to \(\preceq\). By the Alexandroff characterization of \(\tau_P\), we conclude that $\displaystyle{U\in\tau_P}$.
\end{proof}

\bigskip

The following theorem provides a pathwise interpretation of open sets in the Markov accessibility topology. Informally, these are precisely the regions of the state space that can only be visited during an initial portion of a trajectory; once the chain leaves such a set, it can never return to it.

\bigskip

\begin{theorem}[Pathwise characterization of open sets]
\label{T:pathwise}
For a subset $U\subseteq X$, define the exit time from $U$ by
\[
\sigma_U:=\inf\{n\in\mathbb N_0:X_n\notin U\},
\]
Then the following statements are equivalent:

\begin{enumerate}
    \item $U\in\tau_P$;

    \item for every initial distribution $\mu$ on $X$, the process
    \[
    \bigl(\mathbf 1_U(X_n)\bigr)_{n\geq 0}
    \]
    is $\mathbb P_\mu$-almost surely non-increasing;

    \item for every initial distribution $\mu$ on $X$,
    \[
    \mathbb P_\mu
    \bigl(
    \sigma_U<\infty,\,
    \tau_U\circ\theta_{\sigma_U}<\infty
    \bigr)
    =0.
    \]
\end{enumerate}
The random variable
\[
\tau_U\circ\theta_{\sigma_U}
\]
represents the first return time to \(U\) after the exit time \(\sigma_U\). In other words, once the chain exits the set \(U\), it cannot return to \(U\) with positive probability.
\end{theorem}

\begin{proof}
Assume that $U\in\tau_P$. By definition of the Markov primal topology,
\[
P^{*}(U)\subseteq U.
\]
Equivalently,
\[
P(x,U)=0,
\qquad
\forall\,x\in X\setminus U.
\]
Let $\mu$ be any initial distribution. For each $n\geq0$, consider the event
\[
A_n:=\{X_n\notin U,\ X_{n+1}\in U\}.
\]
We claim that
\[
\mathbb P_\mu(A_n)=0.
\]
Indeed,
\[
\mathbb P_\mu(A_n)
=
\mathbb E_\mu
\!\left[
\mathbf 1_{\{X_n\notin U\}}
\mathbf 1_{\{X_{n+1}\in U\}}
\right].
\]
Using the Markov property,
\[
\mathbb P_\mu(A_n)
=
\mathbb E_\mu
\!\left[
\mathbf 1_{\{X_n\notin U\}}
P(X_n,U)
\right].
\]
Since $P(x,U)=0$ for every $x\notin U$, the integrand vanishes identically. Hence
\[
\mathbb P_\mu(A_n)=0
\qquad
\forall\,n\geq0.
\]
Therefore,
\[
\mathbb P_\mu
\!\left(
\bigcup_{n\ge0}A_n
\right)
=
0.
\]
In other words, with probability one the trajectory can never make a transition from
$X\setminus U$ into $U$. Consequently,
$\displaystyle{
\bigl(\mathbf 1_U(X_n)\bigr)_{n\ge0}
}$ is $\mathbb P_\mu$-almost surely non-increasing. Now, we assume that for every initial distribution $\mu$,
$ \displaystyle{
\bigl(\mathbf 1_U(X_n)\bigr)_{n\ge0}
}$ is $\mathbb P_\mu$-almost surely non-increasing. We are going to show that $U\in\tau_P$. It suffices to prove that
\[
P(x,U)=0,
\qquad
\forall\,x\in X\setminus U.
\]
Fix $x\in X\setminus U$ and choose the initial distribution
$\displaystyle{
\mu=\delta_x.
}$
Then
\[
\mathbf 1_U(X_0)=0
\qquad
\mathbb P_{\delta_x}\text{-a.s.}
\]
Since the indicator process is almost surely non-increasing,
\[
\mathbf 1_U(X_1)=0
\qquad
\mathbb P_{\delta_x}\text{-a.s.}
\]
Hence
\[
\mathbb P_{\delta_x}(X_1\in U)=0,
\]
that is, $\displaystyle{P(x,U)=0}$. Since $x\in X\setminus U$ was arbitrary, $\displaystyle{P^{*}(U)\subseteq U}$. Therefore, $\displaystyle{U\in\tau_P}$. Now, we suppose that for every initial distribution $\mu$,
$ \displaystyle{
\bigl(\mathbf 1_U(X_n)\bigr)_{n\ge0}
}$ is $\mathbb P_\mu$-almost surely non-increasing.
Fix an initial distribution $\mu$.
On the event $\{\sigma_U<\infty\}$, the chain exits $U$ at time $\sigma_U$, so $\displaystyle{X_{\sigma_U}\notin U}$. Since the indicator process is non-increasing, once it reaches the value $0$ it can never return to the value $1$. Therefore the shifted trajectory $\displaystyle{(X_{\sigma_U+n})_{n\ge0}}$ never hits $U$. Equivalently,
\[
\tau_U\circ\theta_{\sigma_U}
=
\infty
 \  \ 
\text{on } \ \ \{\sigma_U<\infty\},
\]
almost surely. Hence
\[
\mathbb P_\mu
\bigl(
\sigma_U<\infty,
\,
\tau_U\circ\theta_{\sigma_U}<\infty
\bigr)
=
0.
\]
Finally, assume that for every initial distribution $\mu$,
\[
\mathbb P_\mu
\bigl(
\sigma_U<\infty,
\,
\tau_U\circ\theta_{\sigma_U}<\infty
\bigr)
=
0.
\]
Fix an initial distribution $\mu$. We show that
\[
\bigl(\mathbf 1_U(X_n)\bigr)_{n\ge0}
\]
is $\mathbb P_\mu$-almost surely non-increasing. Suppose, to the contrary, that the indicator process is not non-increasing on some trajectory. Then there exist integers $m<n$ such that
\[
X_m\notin U,
\qquad
X_n\in U.
\]
Let $\sigma_U$ be the first exit time from $U$. Then $\displaystyle{\sigma_U<\infty}$, and there exists a later time at which the chain returns to $U$. Consequently,
\[
\tau_U\circ\theta_{\sigma_U}<\infty.
\]
Thus the trajectory belongs to the event
\[
\{\sigma_U<\infty,
\,
\tau_U\circ\theta_{\sigma_U}<\infty\},
\]
which has probability zero by assumption. Therefore the set of trajectories for which
$\displaystyle{
\bigl(\mathbf 1_U(X_n)\bigr)_{n\ge0}
}$
fails to be non-increasing has probability zero. Hence the process is
$\mathbb P_\mu$-almost surely non-increasing.
\end{proof}

\bigskip

In the next result, we present the minimal open set as positive hitting domains.

\begin{corollary}
For \(x\in X\), let $\displaystyle{\tau_x:=\inf\{n\in\mathbb N_0:X_n=x\}}$ denote the first hitting time of the state \(x\). Then,
\[
P_x(x)
=
\{y\in X:\mathbb P_y(\tau_x<\infty)>0\}.
\]
Precisely, the minimal open neighborhood of \(x\) is precisely the set of states from which \(x\) can be hit with positive probability.
\end{corollary}

\begin{proof}
By definition of the minimal open neighborhood,
\[
P_P(x)=\{y\in X:y\preceq x\}.
\]
Thus, \(y\in P_P(x)\) if and only if $y\preceq x$. By definition of the accessibility relation, this is equivalent to the existence of some
\(n\in\mathbb N_0\) such that $\displaystyle{P^n(y,x)>0}$. On the other hand,
\[
\mathbb P_y(\tau_x<\infty)>0
\]
if and only if the chain, starting from \(y\), reaches \(x\) at some finite time with positive probability. This is equivalent to saying that there exists \(n\in\mathbb N_0\) such that
\[
\mathbb P_y(X_n=x)>0.
\]
Since
\[
\mathbb P_y(X_n=x)=P^n(y,x),
\]
we obtain that
\[
\mathbb P_y(\tau_x<\infty)>0
\]
holds if and only if there exists \(n\in\mathbb N_0\) such that $\displaystyle{P^n(y,x)>0}$. Therefore, $y\in P_P(x)$.
\end{proof}

\section{Communication Quotients and Canonical $T_0$ Representations}

The communication structure of a Markov chain is reflected in a remarkable way by its Markov accessibility topology. Indeed, previous results show that states belonging to the same communication class possess identical minimal open neighborhoods and are therefore indistinguishable from the topological point of view. This suggests that the true topological objects associated with the chain are not individual states, but rather communication classes.

\bigskip

Motivated by this observation, we study the quotient space obtained by identifying states that communicate with each other. We prove that the accessibility preorder descends naturally to the quotient and becomes a partial order, yielding a canonical \(T_0\)-Alexandroff space that represents the global communication structure of the chain.

\bigskip

\begin{definition}
For each state \(x\in X\), the set
\[
[x]
=
\{y\in X : y\sim x\}
\]
is called the \emph{communication class} of \(x\). The quotient set
\[
X/{\sim}
=
\{[x] : x\in X\}
\]
is called the \emph{communication quotient space}.
\end{definition}

\bigskip

Having identified communication classes as the natural topological building blocks of the Markov accessibility topology, it remains to understand how accessibility behaves at the level of these classes. This leads naturally to a relation on the quotient space \(X/{\sim}\), obtained by lifting the accessibility preorder from states to communication classes. Let \([x],[y]\in X/{\sim}\), and consider the relation
\[
[x]\preceq_q [y]
\quad \text{if and only if} \quad
x\preceq y.
\]

The following proposition shows that this relation is well defined, that is, it does not depend on the choice of representatives of the communication classes.

\bigskip

\begin{proposition}
The relation \(\preceq_q\) on \(X/{\sim}\) is well defined.
\end{proposition}
\begin{proof}
Let \([x],[y]\in X/{\sim}\). We need to show that the definition
\[
[x]\preceq_q [y]
\quad \text{if and only if} \quad
x\preceq y
\]
does not depend on the choice of representatives. Suppose that
\[
[x]=[x']
\qquad \text{and} \qquad
[y]=[y'].
\]
This means that
\[
x\sim x'
\qquad \text{and} \qquad
y\sim y'.
\]
Therefore,
\[
x'\preceq x
\qquad \text{and} \qquad
y\preceq y'.
\]
Now assume that $x\preceq y$. Using the transitivity of the accessibility preorder, we obtain
\[
x'\preceq x\preceq y\preceq y'.
\]
Hence,
\[
x'\preceq y'.
\]
\end{proof}

The previous proposition ensures that accessibility induces a relation on the communication quotient. The next result shows that passing to communication classes removes precisely the ambiguity caused by mutual accessibility, so that the induced relation becomes a partial order.

\bigskip

\begin{theorem}
The relation \(\preceq_q\) is a partial order on the communication quotient \(X/{\sim}\).
\end{theorem}

\begin{proof}
We prove that \(\preceq_q\) is reflexive, antisymmetric, and transitive. First, let \([x]\in X/{\sim}\). Since the accessibility preorder \(\preceq\) is reflexive, we have $x\preceq x$. Therefore, by definition of \(\preceq_q\),
\[
[x]\preceq_q [x].
\]
Thus, \(\preceq_q\) is reflexive. Now suppose that $[x]\preceq_q [y]$ and $[y]\preceq_q [x]$. By definition of \(\preceq_q\), this means that $x\preceq y$ and
$y\preceq x$. Hence,
\[
x\sim y.
\]
Therefore,
\[
[x]=[y].
\]
Thus, \(\preceq_q\) is antisymmetric. Finally, suppose that $[x]\preceq_q [y]$ and $[y]\preceq_q [z]$, which means that $x\preceq y$ and $y\preceq z$. Since \(\preceq\) is transitive, it follows that $x\preceq z$. Therefore,
\[
[x]\preceq_q [z].
\]
Thus, \(\preceq_q\) is transitive.
\end{proof}

\bigskip

\begin{corollary}
The Alexandroff topology on \(X/{\sim}\) induced by the partial order \(\preceq_q\) is \(T_0\).
\end{corollary}

\begin{proof}
It is a classical fact that an Alexandroff topology is \(T_0\) if and only if its associated preorder is antisymmetric. Since \(\preceq_q\) is a partial order on \(X/{\sim}\), it is antisymmetric. Therefore, the Alexandroff topology induced by \(\preceq_q\) is \(T_0\).
\end{proof}

\bigskip
Let $\displaystyle{q:X\to X/{\sim}}$ be the canonical projection onto the communication quotient, defined by
\[
q(x)=[x],
\qquad x\in X.
\]
Given a Markov chain \((X_n)_{n\geq0}\) on \(X\), we consider the projected process
\[
Y_n:=q(X_n)=[X_n],
\qquad n\geq0.
\]
Thus, \(Y_n\) records the communication class in which the chain \(X_n\) lies at time \(n\). In this way, the original dynamics on the state space \(X\) induce a process on the quotient space \(X/{\sim}\). 

\bigskip

The quotient construction naturally raises the following question: does the projected process \((Y_n)_{n\geq0}\) define a Markov chain on the communication quotient \(X/{\sim}\)? In general, this is not the case, since the transition probabilities between communication classes may depend on the particular representative of the initial class. The appropriate condition ensuring that the projected process inherits the Markov property is the well-known notion of lumpability.

\bigskip

\begin{definition}
A discrete-time Markov chain \((X,P)\) is said to be \emph{lumpable} with respect to the communication partition if
\[
P(x,[z])=P(y,[z]),
\qquad
\forall\,x,y\in[w],\quad \forall\,[w],[z]\in X/{\sim},
\]
\end{definition}

\bigskip

\begin{theorem}
Let \((X_n)_{n\geq0}\) be a discrete-time Markov chain with transition matrix \(P\), and suppose that the chain is lumpable with respect to the partition of \(X\) into communication classes. Then the following statements hold.

\begin{enumerate}
    \item The process \((Y_n)_{n\geq0}\) is a discrete-time Markov chain on \(X/{\sim}\).

    \item Its transition matrix is given by
    \[
    \overline{P}([x],[y])
    =
    P(x,[y]),
    \]

    where
    \[
    P(x,[y])
    =
    \sum_{z\in[y]}P(x,z).
    \]
   
    \item The accessibility relation of the quotient chain coincides with the partial order \(\preceq_q\).

    \item the Markov accessibility topology associated with the quotient transition matrix \(\overline P\) coincides with the Alexandroff topology induced by the partial order \(\preceq_q\).
\end{enumerate}
\end{theorem}

\begin{proof}
First, the lumpability condition ensures that the quantity
\[
\overline P([x],[y]) := P(x,[y])
\]
is well defined. Indeed, if \(x'\in[x]\), then by lumpability,
\[
P(x,[y])=P(x',[y]).
\]
Moreover,
\[
\overline P([x],[y])\geq 0
\]
and
\[
\sum_{[y]\in X/{\sim}}\overline P([x],[y])
=
\sum_{[y]\in X/{\sim}}P(x,[y])
=
\sum_{z\in X}P(x,z)
=
1.
\]
Hence \(\overline P\) is a stochastic matrix on \(X/{\sim}\). Now let
\[
Y_n=q(X_n)=[X_n].
\]
For communication classes \([x]\) and \([y]\), we have
\[
\mathbb P(Y_{n+1}=[y]\mid Y_n=[x])
=
\mathbb P(X_{n+1}\in[y]\mid X_n\in[x]).
\]
By the lumpability condition, the total probability of moving from the class \([x]\) to the class \([y]\) does not depend on the representative \(x\in[x]\). Therefore,
\[
\mathbb P(Y_{n+1}=[y]\mid Y_n=[x])
=
\overline P([x],[y]).
\]
More generally, conditioning on the whole past of \(Y\), the same argument gives
\[
\mathbb P(Y_{n+1}=[y]\mid Y_n,[Y_{n-1}],\ldots,Y_0)
=
\overline P(Y_n,[y]).
\]
Thus \((Y_n)_{n\geq0}\) is a discrete-time Markov chain on \(X/{\sim}\) with transition matrix \(\overline P\). It remains to identify its accessibility relation. Suppose first that
\[
[x]\preceq_q [y].
\]
By definition, this means that $\displaystyle{x\preceq y}$. Hence there exists \(n\in\mathbb N_0\) such that
\[
P^n(x,y)>0.
\]
Therefore, starting from the class \([x]\), the projected chain can reach the class \([y]\) in \(n\) steps with positive probability. Thus \([y]\) is accessible from \([x]\) for the quotient chain. Conversely, suppose that \([y]\) is accessible from \([x]\) for the quotient chain. Then there exists \(n\in\mathbb N_0\) such that
\[
\overline P^{\,n}([x],[y])>0.
\]
This means that the projected chain can move from \([x]\) to \([y]\) with positive probability in \(n\) steps. Hence there exist representatives along such a path in the original chain, and consequently some state in \([y]\) is accessible from some state in \([x]\). Since all states inside a communication class communicate with each other, it follows that $\displaystyle{x\preceq y}$. Therefore,
\[
[x]\preceq_q [y].
\]
Thus the accessibility relation of the quotient chain coincides with the partial order \(\preceq_q\). Finally, by Theorem~\ref{T1}, the Markov accessibility topology associated with the quotient transition matrix \(\overline P\) is the Alexandroff topology generated by the accessibility relation of the quotient chain. Since this accessibility relation is precisely \(\preceq_q\), the Markov accessibility topology generated by \(\overline P\) coincides with the canonical \(T_0\)-Alexandroff topology on \(X/{\sim}\) induced by the partial order \(\preceq_q\).
\end{proof}

\bigskip

\begin{example}
Consider the discrete-time Markov chain with state space
\[
X=\{1,2,3,4,5,6,7\},
\]
whose transition matrix is
\[
P=
\begin{pmatrix}
\frac12 & \frac12 & 0 & 0 & 0 & 0 & 0\\
\frac13 & \frac13 & \frac13 & 0 & 0 & 0 & 0\\
0 & 0 & \frac12 & \frac12 & 0 & 0 & 0\\
0 & 0 & 0 & \frac12 & \frac12 & 0 & 0\\
0 & 0 & 0 & \frac14 & \frac14 & \frac12 & 0\\
0 & 0 & 0 & 0 & 0 & \frac12 & \frac12\\
0 & 0 & 0 & 0 & 0 & \frac12 & \frac12
\end{pmatrix}.
\]
The communication classes of the chain are
\[
[1]=\{1,2\},\qquad
[3]=\{3\},\qquad
[4]=\{4,5\},\qquad
[6]=\{6,7\}.
\]
The accessibility relation induces the partial order
\[
[1]\preceq_q[3]\preceq_q[4]\preceq_q[6].
\]
Consequently, the communication quotient is the totally ordered set
\[
X/{\sim}
=
\bigl\{
[1],[3],[4],[6]
\bigr\},
\]
whose Hasse diagram is

\[
[1]
\longrightarrow
[3]
\longrightarrow
[4]
\longrightarrow
[6].
\]
The minimal open neighborhoods are given by
\[
P_P(1)=P_P(2)=\{1,2\},
\]

\[
P_P(3)=\{1,2,3\},
\]

\[
P_P(4)=P_P(5)=\{1,2,3,4,5\},
\]

and

\[
P_P(6)=P_P(7)=X.
\]

Hence, the Markov accessibility topology is

\[
\tau_P=
\{
\varnothing,
\{1,2\},
\{1,2,3\},
\{1,2,3,4,5\},
X
\}.
\]

Its closed sets are

\[
\varnothing,
\{6,7\},
\{4,5,6,7\},
\{3,4,5,6,7\},
X,
\]

which coincide with the absorbing subsets of the chain. Moreover, the partition into communication classes is lumpable. Therefore, the projected process $\displaystyle{Y_n=[X_n]}$ is itself a Markov chain on the quotient space \(X/{\sim}\), whose transition matrix is

\[
\overline P=
\begin{pmatrix}
\frac23 & \frac13 & 0 & 0\\
0 & \frac12 & \frac12 & 0\\
0 & 0 & \frac12 & \frac12\\
0 & 0 & 0 & 1
\end{pmatrix}.
\]

\end{example}

\section{Conclusions and Future Directions}

In this paper, we introduced the Markov Accessibility Topology, a canonical
topological structure associated with the accessibility relation of a
discrete-time Markov chain. By showing that this topology coincides with the
Alexandroff topology induced by the accessibility preorder, we established a
direct connection between probabilistic reachability, order theory, and
topology. Within this framework, several classical notions from Markov chain theory
admit natural topological interpretations. In particular, absorbing subsets
are precisely the closed sets of the Markov accessibility topology, while
irreducibility is equivalent to the triviality of the associated topology.
Moreover, minimal open neighborhoods encode the states from which a given
state is accessible, and they can be characterized in terms of positive
hitting probabilities. At the pathwise level, open sets correspond to regions
that cannot be revisited after they have been left. Finally, the quotient by
communication classes yields a canonical \(T_0\)-Alexandroff space whose
specialization order describes the global accessibility structure of the
chain. Under suitable lumpability assumptions, this quotient also supports a
well-defined Markov dynamics.

These results show that the accessibility relation is not merely an auxiliary
notion used to describe possible transitions. Rather, it determines an
intrinsic topological structure that organizes several qualitative properties
of the underlying Markov chain. The present work therefore provides a starting
point for a broader investigation of the interaction between stochastic
dynamics, order-theoretic topology, and lattice theory. Two distinct research directions naturally emerge from this construction. The first concerns the separation properties and order-theoretic invariants of
the Markov Accessibility Topology. Several separation axioms lying between
\(T_0\) and \(T_1\), such as \(T_D\), \(T_{1/2}\), and other generalized
separation conditions, admit formulations in terms of the specialization
preorder of a topological space. Since, in the present setting, this preorder
is precisely the accessibility relation, it is natural to seek probabilistic
characterizations of these axioms in terms of communication classes,
absorbing states, accessibility between distinct states, and the structure of
the communication quotient. It would also be interesting to investigate
order-theoretic invariants such as the height, width, and maximal chains of the
quotient partially ordered set. Probabilistically, these invariants may
describe the number and arrangement of successive communication classes that
can be reached in a strictly one-directional manner. This direction may lead
to a systematic classification of Markov chains according to the separation
and order properties of their accessibility topologies.

A second, independent direction concerns complements of the Markov
Accessibility Topology in the lattice of topologies. The existence and
structure of complements and maximal complements of preorders and topologies
have been studied from a lattice-theoretic perspective; see, for instance,
\cite{Uzcategui}. Since the topology introduced in this paper is generated by
an accessibility preorder, it is natural to ask whether a complement of this
topology can itself be realized as the accessibility topology of another
Markov chain on the same state space. This raises the possibility of defining
complementary Markov dynamics, namely pairs of chains whose accessibility
structures are complementary in an appropriate lattice-theoretic sense.
Relevant questions include the existence and nonuniqueness of such
complementary chains, the realization of maximal complements by transition
matrices, and the relationships between the communication classes, absorbing
subsets, and order invariants of the two dynamics. The study of these
complementary structures may provide a new connection between Markov chain
theory and the lattice theory of preorders and topologies.

\newpage


\end{document}